\documentclass[11pt]{amsart}

\author[A.~J.~Di~Scala]{Antonio J.~Di~Scala}
\address{\parbox{\linewidth}{
Politecnico di Torino, Department of Mathematical Sciences\\
Corso Duca degli Abruzzi 24, 10129 Torino, Italy\\[-8pt]}}
\email{antonio.discala@polito.it}

\author[C.~Sanna]{Carlo Sanna}
\address{\parbox{\linewidth}{
Politecnico di Torino, Department of Mathematical Sciences\\
Corso Duca degli Abruzzi 24, 10129 Torino, Italy\\[-8pt]}}
\email{carlo.sanna.dev@gmail.com}

\author[E.~Signorini]{Edoardo Signorini}
\address{\parbox{\linewidth}{
Universit\`a di Trento, Department of Mathematics\\
Via Sommarive 14, 38123 Trento, Italy\\[1pt]
Telsy Elettronica e Telecomunicazioni S.p.A.\\
Corso Svizzera 185, 10149 Torino, Italy\\[-8pt]}}
\email{edoardo.signorini@telsy.com}



\title{On the condition number of the Vandermonde matrix of the $n$th cyclotomic polynomial}

\usepackage{amsmath}
\usepackage{amssymb}
\usepackage{amsthm}
\usepackage{geometry}
\usepackage{color}
\usepackage{hyperref}
\usepackage{enumerate}
\hypersetup{colorlinks=true}

\newtheorem{thm}{Theorem}[section]

\newtheorem{lem}[thm]{Lemma}

\theoremstyle{remark}

\newcommand{\Cond}{\operatorname{Cond}}
\newcommand{\rad}{\operatorname{rad}}
\newcommand{\Tr}{\operatorname{Tr}}
\newcommand{\Id}{\operatorname{Id}}

\begin{document}

\maketitle

\begin{abstract}
Recently, Blanco-Chac{\'o}n proved the equivalence between the Ring Learning With Errors and Polynomial Learning With Errors problems for some families of cyclotomic number fields by giving some upper bounds for the condition number $\Cond(V_n)$ of the Vandermonde matrix $V_n$ associated to the $n$th cyclotomic polynomial.
We prove some results on the singular values of $V_n$ and, in particular, we determine $\Cond(V_n)$ for $n = 2^k p^\ell$, where $k, \ell \geq 0$ are integers and $p$ is an odd prime number.
\end{abstract}

\section{Introduction}

Let $n$ be a positive integer and let $\zeta_1, \dots, \zeta_m$ be the primitive $n$th roots of unity, where $m := \varphi(n)$ is the Euler's totient function of $n$.
Moreover, let $V_n$ be the Vandermonde matrix associated with the $n$th cyclotomic polynomial, that is,
\begin{equation*}
V_n := \begin{pmatrix}
1 & \zeta_1 & \zeta_1^2 & \cdots & \zeta_1^{m-1} \\
1 & \zeta_2 & \zeta_2^2 & \cdots & \zeta_2^{m-1} \\
1 & \zeta_3 & \zeta_3^2 & \cdots & \zeta_3^{m-1} \\
\vdots & \vdots & \vdots & \ddots & \vdots \\
1 & \zeta_m & \zeta_m^2 & \cdots & \zeta_m^{m-1} \\
\end{pmatrix} .
\end{equation*}
Recall that the \emph{condition number} of an invertible complex matrix $A = (a_{i,j})_{1 \leq i,j \leq k}$ is defined~by
\begin{equation*}
\Cond (A) := \|A\|\|A^{-1}\| ,
\end{equation*}
where
\begin{equation*}
\|A\| := \sqrt{\sum_{i \,=\, 1}^k \sum_{j \,=\, 1}^k |a_{i,j}|^2} = \sqrt{\Tr(A^* A)} 
\end{equation*}
is the \emph{Frobenius norm} of $A$ and $A^*$ is the conjugate transpose of $A$.

Recently, Blanco-Chac{\'o}n~\cite{Blanco} gave some upper bounds for the condition number of $V_n$.
This in order to prove the equivalence between the \emph{Ring Learning With Errors} and \emph{Polynomial Learning With Errors} problems for some infinite families of cyclotomic number fields~(see also~\cite{MR2980590, MR2660480, MR3794783}).

Our first result is the following.

\begin{thm}\label{thm:tosquarefree}
For every positive integer $n$, we have
\begin{equation*}
\Cond(V_n) = \frac{n}{\rad(n)} \Cond(V_{\rad(n)}) ,
\end{equation*}
where $\rad(n)$ denotes the product of all prime factors of $n$.
\end{thm}

Our second result is a formula for the condition number of $V_n$ when $n$ is a prime power or a power of $2$ times an odd prime power.

\begin{thm}\label{thm:2primepower}
If $n = p^k$, where $k$ is a positive integer and $p$ is a prime number, or if $n = 2^k p^\ell$, where $k, \ell$ are positive integers and $p$ is an odd prime number, then
\begin{equation*}
\Cond(V_n) = \varphi(n)\sqrt{2\left(1 - \frac1{p}\right)} .
\end{equation*}
\end{thm}

In particular, Theorem~\ref{thm:2primepower} improves the upper bound $\Cond(V_n) \leq 2(p-1)\varphi(n)$ given by Blanco-Chac{\'o}n in the case in which $n = p^k$ is a prime power~\cite[Theorem~3.9]{Blanco}.

Our proofs of Theorems~\ref{thm:tosquarefree} and \ref{thm:2primepower} are based on the study of the Gram matrix $G_n := V_n^*\, V_n$.
Regarding that, we give also the following result.

\begin{thm}\label{thm:nGn1}
For every positive integer $n$, the matrix $n\,G_n^{-1}$ has integer entries.
\end{thm}

From a number-theoretic point of view, it might be of some interest trying to describe the entries of $n\,G_n^{-1}$ explicitely, or at least understand the integer sequence $\Tr(n\,G_n^{-1})_{n \geq 1}$ (which is related to $\Cond(V_n)$ by \eqref{equ:condmusigma} below).

\subsection*{Acknowledgments}
A.~J.~Di~Scala and C.~Sanna are members of GNSAGA of INdAM and of CrypTO, the group of Cryptography and Number Theory of Politecnico di Torino.
E.~Signorini is supported by Telsy S.p.A.

\section{Proofs}

For every positive integer $n$, the \emph{Ramanujan's sums} modulo $n$ are defined by
\begin{equation*}
c_n(t) := \sum_{i \,=\, 1}^m \zeta_i^t ,
\end{equation*}
for all integers $t$.
It is easy to check that $c_n(\cdot)$ is an even periodic function with period $n$.
Moreover, the following formula holds~\cite[Theorem~272]{MR2445243}
\begin{equation}\label{equ:vonSterneck}
c_n(t) = \mu\!\left(\tfrac{n}{(n, t)}\right) \frac{\varphi(n)}{\varphi\!\left(\tfrac{n}{(n, t)}\right)} ,
\end{equation}
where $\mu$ is the M\"obius function and $(n, t)$ denotes the greatest common divisor of $n$ and $t$.

Let $G_n := V_n^* \, V_n$ be the \emph{Gram matrix} of $V_n$.
By the previous considerations, we have
\begin{equation}\label{equ:Gn}
G_n = {\small\begin{pmatrix}
c_n(0) & c_n(1) & c_n(2) & \cdots & c_n(m - 1) \\
c_n(1) & c_n(0) & c_n(1) & \cdots & c_n(m - 2) \\
c_n(2) & c_n(1) & c_n(0) & \cdots & c_n(m - 3) \\
\vdots & \vdots & \vdots & \ddots & \vdots \\
c_n(m - 1) & c_n(m - 2) & c_n(m - 3) & \cdots &	 c_n(0) \\
\end{pmatrix}}
= \big(c_n(i - j)\big)_{1 \leq i,j \leq m} .
\end{equation}
In particular, $G_n$ is a symmetric Toeplitz matrix with integer entries.

Let $\sigma_1, \dots, \sigma_s$ be the distinct eigenvalues of $G_n$, which are real and positive, since $G_n$ is the Gram matrix of an invertible matrix, and let $\mu_1, \dots, \mu_s$ be their respective multiplicities.
We~have
\begin{equation}\label{equ:condmusigma}
\Cond(V_n) = \|V_n\|\|V_n^{-1}\| = m\sqrt{\Tr(G_n^{-1})} = m\sqrt{\sum_{i \,=\, 1}^s \frac{\mu_i}{\sigma_i}} .
\end{equation}
Therefore, the study of $\Cond(V_n)$ is equivalent to the study of the eigenvalues of $G_n$.

The next lemma relates the characteristic polynomials of $G_n$ and $G_{\rad(n)}$.

\begin{lem}\label{lem:Gneigen}
For every positive integer $n$, we have
\begin{equation*}
\det(G_n - x\Id_m) = h^m \det\!\big(G_{n^\prime} - \tfrac{x}{h}\Id_{m^\prime}\!\big)^h ,
\end{equation*}
where $n^\prime := \rad(n)$, $m^\prime := \varphi(n^\prime)$, and $h := n / n^\prime$.
\end{lem}
\begin{proof}
We know from~\eqref{equ:Gn} that $G_n = \big(c_n(i - j)\big)_{0 \leq i, j < m}$, where we shifted the indices $i,j$ to the interval ${[0, m)}$ since this does not change the differences $i - j$ and simplifies the next arguments.
Write the integers $i, j \in {[0, m)}$ in the form $i = h i^\prime + i^{\prime\prime}$ and $j = h j^\prime + j^{\prime\prime}$, where $i^\prime, j^\prime \in {[0, m^\prime)}$ and $i^{\prime\prime}, j^{\prime\prime} \in {[0, h)}$ are integers.
By~\eqref{equ:vonSterneck} we have that $c_n(i - j) \neq 0$ if and only if $h$ divides $i - j$ (otherwise, $n / (n, i - j)$ is not squarefree), which in turn happens if and only if $i^{\prime\prime} = j^{\prime\prime}$.
In such a case, we have $(n, i - j) = h (n^\prime, i^\prime - j^\prime)$ and, again by~\eqref{equ:vonSterneck}, it follows that
\begin{equation*}
c_n(i - j) = \mu\!\left(\tfrac{n}{(n, i - j)}\right) \frac{\varphi(n)}{\varphi\!\left(\tfrac{n}{(n, i - j)}\right)} = \mu\!\left(\tfrac{n^\prime}{(n^\prime, i^\prime - j^\prime)}\right) \frac{h\, \varphi(n^\prime)}{\varphi\!\left(\tfrac{n^\prime}{(n^\prime, i^\prime - j^\prime)}\right)} = h\, c_{n^\prime}(i^\prime - j^\prime) .
\end{equation*}
Therefore, we have found that $G_n$ consists of $m^\prime \times m^\prime$ diagonal blocks of sizes $h \times h$.
Precisely,
\begin{equation*}
G_n = h \big(c_{n^\prime}(i^\prime - j^\prime) \Id_h \big)_{0 \leq i^\prime, j^\prime < m^\prime} = h\, G_{n^\prime} \otimes \Id_h ,
\end{equation*}
where $\otimes$ denotes the Kronecker product.
Consequently, the characteristic polynomial of $G_n$ is
\begin{align*}
\det(G_n - x\Id_m) &= h^m \det\big(G_{n^\prime} \otimes \Id_h - \tfrac{x}{h}\Id_m\!\big) \\
&= h^m \det\!\big((G_{n^\prime} - \tfrac{x}{h}\Id_{m^\prime}) \otimes \Id_h\!\big) \\
&= h^m \det\!\big(G_{n^\prime} - \tfrac{x}{h}\Id_{m^\prime}\!\big)^h ,
\end{align*}
as claimed.
\end{proof}

Now we are ready to prove the first result.

\subsection{Proof of Theorem~\ref{thm:tosquarefree}}

Let $n^\prime := \rad(n)$, $m^\prime := \varphi(n^\prime)$, and $h := n / n^\prime$.
Furthermore, let $\sigma_1^\prime, \dots, \sigma_{s^\prime}^\prime$ be the distinct eigenvalues of $G_{n^\prime}$, with respective multiplicities $\mu_1^{\prime}, \dots, \mu_{s^\prime}^{\prime}$.
It~follows from Lemma~\ref{lem:Gneigen} that $s^\prime = s$ and that the eigenvalues of $G_n$ are $h \sigma_1^\prime, \dots, h \sigma_{s}^\prime$, with respective multiplicities $h \mu_1^\prime, \dots, h \mu_{s}^\prime$.
Hence, \eqref{equ:condmusigma} yields
\begin{equation*}
\Cond(V_n) = m\sqrt{\sum_{i \,=\, 1}^s \frac{\mu_i}{\sigma_i}} = m\sqrt{\sum_{i \,=\, 1}^s \frac{\mu_i^\prime}{\sigma_i^\prime}} = \frac{m}{m^\prime}\Cond(V_{n^\prime}) = \frac{n}{n^\prime}\Cond(V_{n^\prime}) ,
\end{equation*}
as claimed.~$\square$\\

We need a couple of preliminary lemmas to the proof of Theorem~\ref{thm:2primepower}.

\begin{lem}\label{lem:G2n}
For every odd positive integer $n$, the matrices $G_{2n}$ and $G_n$ have the same eigenvalues (with the same multiplicities).
\end{lem}
\begin{proof}
It is known~\cite[Theorem~67]{MR2445243} that Ramanujan's sums are multiplicative functions respect to their moduli, that is, $c_{ab}(t) = c_a(t)\, c_b(t)$ for all coprime positive integers $a,b$.
Moreover, it is easy to check that $c_2(t) = (-1)^t$.
Thus, \eqref{equ:Gn} gives
\begin{equation*}
G_{2n} = \big(c_{2n}(i-j)\big)_{1\leq i, j \leq m} = \big((-1)^{i-j} c_n(i-j)\big)_{1\leq i, j \leq m} = J^{-1} G_n J ,
\end{equation*}
where $J$ is the $m \times m$ matrix alternating $+1$ and $-1$ on its diagonal and having zeros in all the other entries.
Therefore, $G_n$ and $G_{2n}$ are similar and consequently they have the same eigenvalues.
\end{proof}

\begin{lem}\label{lem:abdiagonal}
Given two complex numbers $a$ and $b$, the determinant of the $k \times k$ matrix
\begin{equation*}
\small\begin{pmatrix}
a & b & b & \cdots & b \\
b & a & b & \cdots & b \\
b & b & a & \cdots & b \\
\vdots & \vdots & \vdots & \ddots & \vdots \\
b & b & b & \cdots & a
\end{pmatrix}
\end{equation*}
is equal to $(a - b)^{k - 1} (a + (k - 1)b)$.
\end{lem}
\begin{proof}
Subtracting the last row from all the other rows, and then adding to the last column all the other columns, the matrix becomes
\begin{equation*}
\tiny\begin{pmatrix}
a - b & 0 & \cdots & 0 & 0 \\
0 & a - b & \cdots & 0 & 0 \\
\vdots & \vdots & \ddots & \vdots & \vdots \\
0 & 0 & 0 & a - b & 0 \\
b & b & b & b & a + b(k - 1)
\end{pmatrix} .
\end{equation*}
Laplace expansion along the last column gives the desired result.
\end{proof}

\subsection{Proof of Theorem~\ref{thm:2primepower}}

First, let us consider $n = p^k$, where $k$ is a positive integer and $p$ is a prime number.
It follows from~\eqref{equ:vonSterneck} that $c_p(t) = p - 1$ if $p$ divides $t$, while $c_p(t) = -1$ otherwise.
Hence, using Lemma~\ref{lem:abdiagonal}, we have
\begin{equation*}
\det(G_p - x\Id_{p-1}) = {\tiny\begin{pmatrix}
p - 1 - x & -1 & \cdots & -1 \\
-1 & p - 1 - x & \cdots & -1 \\
\vdots & \vdots & \ddots & \vdots \\
-1 & -1 & \cdots & p - 1 - x
\end{pmatrix}} 
= (p - x)^{p - 2}\,(1 - x) ,
\end{equation*}
so that the eigenvalues of $G_p$ are $p$ and $1$, with respective multiplicities $p - 2$ and $1$.

As~a~consequence, \eqref{equ:condmusigma} gives
\begin{equation}\label{equ:CondVp}
\Cond(V_p) = (p - 1)\sqrt{2\left(1 - \frac1{p}\right)} ,
\end{equation}
and, thanks to Theorem~\ref{thm:tosquarefree}, we obtain
\begin{equation*}
\Cond(V_{p^k}) = p^{k-1}\Cond(V_p) = p^{k-1}(p - 1)\sqrt{2\left(1 - \frac1{p}\right)} = \varphi(n)\sqrt{2\left(1 - \frac1{p}\right)} ,
\end{equation*}
as claimed.

Now assume that $n = 2^k p^\ell$, where $k, \ell$ are positive integers and $p$ is an odd prime number.
From Lemma~\ref{lem:G2n} and~\eqref{equ:condmusigma} it follows at once that $\Cond(V_{2p}) = \Cond(V_p)$.
Hence, Theorem~\ref{thm:tosquarefree} and~\eqref{equ:CondVp} yield
\begin{equation*}
\Cond(V_{2^k p^\ell}) = 2^{k-1}p^{\ell-1}\Cond(V_{2p}) = 2^{k-1} p^{\ell-1}(p - 1)\sqrt{2\left(1 - \frac1{p}\right)} = \varphi(n)\sqrt{2\left(1 - \frac1{p}\right)} ,
\end{equation*}
as claimed.~$\square$\\

The next lemma is the well known orthogonality relation between the roots of unity.

\begin{lem}\label{lem:ortho}
We have
\begin{equation*}
\sum_{\ell = 1}^n \big(\zeta_k \overline{\zeta_h}\big)^\ell
= \begin{cases}
n & \text{ if } k = h , \\
0 & \text{ if } k \neq h , \\
\end{cases}
\end{equation*}
for $k,h=1,\dots,m$.
\end{lem}

\subsection{Proof of Theorem~\ref{thm:nGn1}}

Let $V_n^{-1} = (w_{i,j})_{1 \leq i, j \leq m}$ and define
\begin{equation*}
S_{i,\ell} := \sum_{k = 1}^m w_{i,k} \zeta_k^{\ell} ,
\end{equation*}
for all integers $i,\ell$ with $1 \leq i \leq m$ and $\ell \geq 0$.
On the one hand, since $V_n^{-1} V_n = \Id_m$, for $\ell < m$ we have that $S_{i,\ell} = \delta_{i,\ell+1}$ (Kronecker delta).
On the other hand, since $\zeta_1, \dots, \zeta_k$ are conjugate algebraic integers with minimal polynomial of degree $m$, for $\ell \ge m$ there exist integers $b_0,\dots,b_{m-1}$ such that $\zeta_k^\ell = b_0 + b_1 \zeta_k + \cdots + b_{m-1} \zeta_k^{m-1}$ for $k=1,\dots,m$, and consequently $S_{i,\ell} = b_0 S_{i,0} + b_1 S_{i,1} + \cdots + b_{m-1} S_{i,m - 1}$.
Hence, $S_{i,\ell}$ is always an integer.

Recalling that $G_n = V_n^*\,V_n$, we have $G_n^{-1} = V_n^{-1} \big(V_n^{-1}\big)^*$.
Hence, also using Lemma~\ref{lem:ortho}, the $(i,j)$ entry of $nG_n^{-1}$ is equal to
\begin{equation*}
n \sum_{k = 1}^m w_{i,k} \overline{w_{j,k}} = \sum_{k = 1}^m \sum_{h = 1}^m w_{i,k} \overline{w_{j,h}} \sum_{\ell = 1}^n \big(\zeta_k \overline{\zeta_h}\big)^\ell = \sum_{\ell = 1}^n \left(\sum_{k = 1}^m w_{i,k} \zeta_k^\ell\right)\overline{\left(\sum_{h = 1}^m w_{j,h} \zeta_h^\ell\right)} = \sum_{\ell = 1}^n S_{i,\ell} S_{j,\ell} ,
\end{equation*}
which is an integer.~$\square$

\bibliographystyle{amsplain}

\end{document}